\documentclass[12pt, 14paper,reqno]{amsart}
\usepackage{amsmath,amsfonts,amssymb}

\theoremstyle{plain}
\newtheorem{thm}{Theorem}
\newtheorem{lemma}{Lemma}
\newtheorem{prop}{Proposition}
\newtheorem{cs}{Case}

\theoremstyle{definition}

\theoremstyle{remark}

\numberwithin{equation}{section}
\numberwithin{lemma}{section}
\numberwithin{thm}{section}
\usepackage{amsmath}
\usepackage{amsfonts}   
\usepackage{amssymb}
\usepackage{amssymb, amsmath, amsthm}
\usepackage[breaklinks]{hyperref}

\usepackage{graphicx}
\usepackage{amsthm}

\begin{document} 

\title{ A note on the class number of certain real cyclotomic fields}
\author{Om Prakash}
\address{Kerala School of Mathematics, Kozhikode - 673571, Kerala, India.}
\email{omprakash@ksom.res.in}
\maketitle

\begin{abstract} 
We construct an infinite family of real cyclotomic fields with non-trivial class group. This result generalizes the result in \cite{Kalyan} in the sense that our family includes theirs.
\end{abstract}
\section{Introduction}
The ring of rational integers $\mathbb{Z}$ has unique factorization of integers into primes. But $\mathcal{O}_{K}$, the ring of integers for any number field $K$ does not necessarily have the property of unique factorization. The extent of failure of unique factorization is measured by the ideal class group which is a finite abelian group and its cardinality is known as the class number of $K$. Although class number is the most fundamental invariant of a number field, there is still a lot of mystery around class numbers. For instance, a well known conjecture states that there exist infinitely many number fields with class number $1$, or even with bounded class number. The `class number one problem' for real quadratic fields is a `folklore conjecture'.

The class number of cyclotomic fields have been studied extensively since  the time mathematicians established linkage between Fermat's last theorem and the unique factorization properties of cyclotomic integers more than a century ago. The class number $h_n$ of cyclotomic field $\mathbb{Q}({\zeta}_n)$ can be written as the product $$h_n={h_n}^+{h_n}^-$$
where ${h_n}^+$ denotes the class number of  real cyclotomic field  $\mathbb{Q}({\zeta}_n)^+$ ($:=\mathbb{Q}({\zeta}_n+{{\zeta}_n}^{-1})$) of cyclotomic field $\mathbb{Q}({\zeta}_n)$.

The class number of real cyclotomic field (also known as the `plus part') remains one of the most enigmatic object to study.

The `minus part' ${h_n}^-$, which is also known as the relative class number, is defined by the quotient $$
{h_n}^- = \frac{h_n}{{h_n}^+}.
$$
This is a tautology up to this point, but it is not tautological to note that ${h_n}^-$ can be calculated quite explicitly by using the class number formula. On applying the class number formula to $\mathbb{Q}({\zeta}_n)$ and $\mathbb{Q}({\zeta}_n)^+$ and than taking the quotient, leaves us with the expression, 
$$
{h_n}^- = \mathcal{Q}m \prod_{\chi \hspace{0.3cm}\text{odd}}{(\frac{-1}{2}B_{1,\chi})}.
$$
Here the product runs over odd Dirichlet character $\chi$; $B_{1,\chi}$ denotes the generalised Bernoulli numbers and 
$$
\mathcal{Q} = \begin{cases}
1 & {\rm ~if~} n\hspace{0.1cm} \text{is a power of prime} \\
2 & {\rm ~otherwise.}
\end{cases}
$$
Despite the fact that ${h_n}^-$ grows exponentially with $n$; ${h_n}^-$ can be directly computed.

On the other hand the `plus part' is still a mystery. Their Minkowski bounds are too large to be effective for the fields with large $n$, and their discriminant are also far to large for Odlyzko's discriminant bound to treat their class number.

Ankeny et al. \cite{ac} showed that, if $p=(2nq)^2+1$ is a prime where $q$ is also a prime and $n>1$ is an integer then $h_p^+>1$. Lang \cite{lang} has also proved the same result for $p=\{(2n+1)q\}^2+4$ where $q$ is a prime and $n\geq 1$ is an integer. In 1987, Osada \cite{osada} generalised \cite{ac} and \cite{lang}.

In a recent work Mishra et al. \cite{Mohit} showed the existence of infinitely many real cyclotomic fields with large class number. Their proof is existential. On the other hand Chakraborty and Hoque \cite{Kalyan} has given explicit families of  real cyclotomic fields having class number strictly bigger than $1$. More precisely they proved the following:

\begin{thm}{\bf(K.Chakraborty and A. Hoque)}
\begin{itemize}
    \item[1.] Let $p\equiv 1 \pmod 4$ and $n\geq 1$ an integer. If $d=(2np)^2-1$, then $h_{4d}^+>1$.
    \item[2.] Let $p\equiv \pm 1 \pmod 4$ be a prime and $n\geq 1$ an integer. If $d=(2np)^2+3$ and $n$ a multiple of $3$, then $h_{4d}^+>1$.
    \item[3.] Let $p\equiv \pm 1 \pmod 8$ be a prime and $n\geq 1$ an integer. If $d=((2n+1)p)^2+2$, then $h_{4d}^+>1$.
    \item[4.] Let $p\equiv 1,3 \pmod 8$ be a prime not equal to 3 and $n\geq 1$ an integer. If $d=((2n+1)p)^2-2$, then $h_{4d}^+>1$.
\end{itemize}
\end{thm}

Their idea was to first prove that the equation $x^2-dy^2=\pm p$ has no solution for $d$ of the form 
$$
d = \begin{cases}
(2np)^2-1 \hspace{0.99cm} {\rm~with~}\hspace{0.2cm} p\equiv 1 \pmod 4\\
(2np)^2+3 \hspace{0.99cm}{\rm~with~}\hspace{0.2cm} p\equiv \pm 1 \pmod 4\\
((2n+1)p)^2+2 {\rm~with~}\hspace{0.2cm} p\equiv \pm 1 \pmod 8\\
((2n+1)p)^2-2 {\rm~with~}\hspace{0.2cm} p\equiv 1,3 \pmod 8\\
\end{cases}$$
where $p$ is a prime and $n$ is a positive integer. Then they lift the class group of the quadratic field to the real cyclotomic field using the following result of Yamaguchi \cite{yama}.

\begin{lemma}\label{l2}
Let $h$ be the class number of quadratic field $\mathbb{Q}(\sqrt{d})$, if $d>0$ is an integer with $\phi(d) > 4$, where $\phi$ is Euler's totient function, then $h|{h_{4d}}^+$.
\end{lemma}

In this article, we obtain an infinite family of $d$'s with non-trivial class number for real cyclotomic fields (which contains \cite{Kalyan}).The main result of this note is: 

\begin{thm} \label{t11}
If $d=\alpha n^2 + \beta n +\gamma$ is a square-free positive integer, then ${h_{4d}}^+>1$.
\end{thm}

Here $\alpha, \beta$ and $\gamma$ are as in Proposition \ref{p1}.
\section{Preliminary}
Let us recall that the continued fraction expression for $\sqrt{d}$ where $d$ is a square-free integer has the form 
$$
\sqrt{d}=[a_0,\overline{a_1,a_2,...,a_{t-1},a_t=2a_0}]
$$ 
where $a_1,a_2,...,a_{t-1}$ is a symmetric sequence. 

Moreover, the converse of the above fact is also true. The following proposition describes the converse.
\begin{prop} \cite{kala} \label{p1}
Let $a_1,a_2,...,a_{t-1}$ ($t\geq 1$) be any symmetric sequence of positive integers. Define a sequence $q_i$ for $-1\leq i \leq t$ by the recurrence relation $q_{i+1} = a_{i+1}q_i + q_{i-1}$, $q_{-1} = 0$, $q_0 = 1$. Then the equation
\begin{equation}\label{e1}
\sqrt{d}=[z,\overline{a_1,a_2,...,a_{t-1},2z]}, z= \lfloor d \rfloor    
\end{equation}
has infinitely many square-free positive solutions $d\equiv 2,3 \pmod 4$ if and only if $q_{t-1}$ is odd or $q_{t-1}$ and $q_{t-2}q_{t-3}$ are both even.
When either of the condition is satisfied then all the solutions of \eqref{e1} are given by 
$$
d=d(n)=\alpha n^2 + \beta n +\gamma ;~~ z(n)=\eta n + \mu,
$$ 
for integers $n\geq 1$. Here $\alpha, \beta, \gamma, \eta, \mu$ depend on $a_1,a_2,...,a_{t-1}$ with $\alpha \neq 0$ is a square; $\eta \neq 0$ and the discriminant is 
$$
{\beta}^2-4\alpha \gamma = (-1)^t~ \text{or} ~ 4(-1)^t ~\text{or} \hspace{0.2cm} 16(-1)^t.
$$
Moreover, if $d(n)$ is square-free and $d\equiv 2,3 \pmod 4$, the fundamental unit of $\mathbb{Q}(\sqrt{d})$ has the form 
\begin{equation}\label{e22}
\epsilon = p(n)+q\sqrt{d}
\end{equation} 
where $p$ is a linear polynomial, $q=q_{t-1}$ and both $p$ and $q$ depend on the sequence $a_1,a_2,...,a_{t-1}.$
\end{prop}
In our case, we consider $t$ to be even so that the norm of the fundamental unit, $Nm_{{\mathbb{Q}(\sqrt{d})}/\mathbb{Q}}(\epsilon)=1$. Also $a_1,a_2,...,a_{s-1}$ be any symmetric sequence of positive integers such that $q_{t-1}$ is odd. Note that such class of symmetric sequences will be non-empty because if we take $t=2$ and consider symmetric sequence $a_1=1$, in this case $q_1=1$ is odd. Hence by Proposition \ref{p1}, the equation 
$$
\sqrt{d}=[z,\overline{a_1,a_2,...,a_{s-1},2z]}, z= \lfloor d \rfloor 
$$
has infinitely many square-free positive integer solutions $d\equiv 3 \pmod 4$. For example, if we take the same sequence $a_1=1$ as above then we can see that there are infinitely many $d$'s with $d\equiv 3 \pmod 4$.  We aren't interested in $d\equiv 2 \pmod 4$ for reasons which would be clear when we prove non-triviality of the class number. Therefore for us $d$ has the form:  
$$
d=\alpha n^2 + \beta n +\gamma.
$$ 
If $d$ is square-free, the fundamental unit of $\mathbb{Q}(\sqrt{d})$ is of the form: 
$$
\epsilon = p(n)+q\sqrt{d}.
$$
\section{Solvability of a Diophantine equation}

\begin{lemma} \label{l1}
The Diophantine equation
$$
x^2-dy^2=\pm p
$$
has no integer solutions for $d=\alpha n^2 + \beta n +\gamma$.
\end{lemma}

\begin{proof}
First we consider
\begin{equation} \label{e2}
x^2-dy^2=p.
\end{equation}
If possible let us assume that \eqref{e2} has an integer solution $(x_0,y_0)$ and without loss of generality we assume that it is the smallest solution with $y_0\geq 1$.
Then,
\begin{equation} \label{e4}
x_0^2-dy_0^2=p.
\end{equation}
We re-write it as
\begin{equation} \label{e3}
Nm_{{\mathbb{Q}(\sqrt{d})}/\mathbb{Q}}(x_0-y_0\sqrt{d})=p
\end{equation}
and multiply the resulting equation \eqref{e3} by $\epsilon$ given by \eqref{e22}. Thus
$$
p= Nm_{{\mathbb{Q}(\sqrt{d})}/\mathbb{Q}}((x_0-y_0\sqrt{d}) \epsilon).
$$
Which is, 
\begin{eqnarray*}
p &=&Nm_{{\mathbb{Q}(\sqrt{d})}/\mathbb{Q}}((x_0-y_0\sqrt{d})(p(n)+q\sqrt{d})).\\
&=& Nm_{{\mathbb{Q}(\sqrt{d})}/\mathbb{Q}}((p(n)x_0-dqy_0)+(x_0q-p(n)y_0)\sqrt{d}).\\
&=& (p(n)x_0-dqy_0)^2 - (x_0q-p(n)y_0)^2 d.
\end{eqnarray*}
Now by the minimality of $y_0$, we have,
$$
y_0 \leq |x_0q-p(n)y_0|.
$$
\begin{cs} \label{c1}
If $y_0\leq x_0q-p(n)y_0$.

Then,
$$
(1+p(n))y_0\leq x_0q.
$$
We use \eqref{e4} and set $p(n)=x_1n+x_2$. This leads to 
\begin{eqnarray*}
\{(1+x_1n+x_2)^2-q^2d\}y_0^2&\leq& q^2p.\\
\{(1+x_1n+x_2)^2-q^2(\alpha n^2 + \beta n +\gamma)\}y_0^2&\leq& q^2p.
\end{eqnarray*}
Also
\begin{equation} \label{e5}
(x_1^2-q^2\alpha)n^2+(2x_1+2x_1x_2-q^2\beta)n+(x_2^2+2x_2+1-q^2\gamma)\leq \frac{q^2p}{y_0^2}.
\end{equation}
Since $x_1, x_2, p, q, \alpha, \beta, \gamma, y_0 $ are fixed integers, therefore, \eqref{e5} implies that there are finitely many $n's$, which further implies that there are finitely many $d's$. This leads to a contradiction.
\end{cs}
\begin{cs}
If $y_0\leq p(n)y_0-x_0q$.

Then, 
$$
x_0q\leq (p(n)-1)y_0.
$$
Again, using \eqref{e4} and setting $p(n)=x_1n+x_2$, we get, 
$$
q^2p\leq \{(x_1n+x_2-1)^2-q^2d\}y_0^2.
$$
This leads to  
$$
q^2p\leq \{(x_1n+x_2-1)^2-q^2(\alpha n^2 + \beta n +\gamma)\}y_0^2
$$
and
$$
q^2(\alpha n^2 + \beta n +\gamma)-(x_1n+x_2-1)^2 \leq \frac{-q^2p}{y_0^2}.
$$
Thus
\begin{equation}
(q^2\alpha-x_1^2)n^2+(q^2\beta-2x_1-2x_1x_2)n+(q^2\gamma-x_2^2-2x_2-1)\leq \frac{-q^2p}{y_0^2}.
\end{equation}
By the same argument as in case\eqref{c1}, we have a contradiction to the fact that there are infinitely many $d's$.
\end{cs}
Next, we consider
\begin{equation} \label{e6}
x^2-dy^2=-p.
\end{equation}
Suppose that \eqref{e6} has an integer solution. Without loss of generality assume that $(x_0,y_0)$ is the smallest solution with $y_0\geq 1$.
     
Then,
\begin{equation} \label{e8}
x_0^2-y_0^2=-p.
\end{equation}
We re-write
\begin{equation} \label{e7}
Nm_{{\mathbb{Q}(\sqrt{d})}/\mathbb{Q}}(x_0-y_0\sqrt{d})=-p.
\end{equation}
and multiply the resulting equation \eqref{e7} by the fundamental unit $\epsilon$.
   
Thus, $$-p= Nm_{{\mathbb{Q}(\sqrt{d})}/\mathbb{Q}}((x_0-y_0\sqrt{d}) \epsilon)$$ Which is, 
\begin{eqnarray*}
-p &=& Nm_{{\mathbb{Q}(\sqrt{d})}/\mathbb{Q}}((x_0-y_0\sqrt{d})(p(n)+q\sqrt{d})).\\
&=& Nm_{{\mathbb{Q}(\sqrt{d})}/\mathbb{Q}}((p(n)x_0-dqy_0)+(x_0q-p(n)y_0)\sqrt{d}).\\
&=& (p(n)x_0-dqy_0)^2 - (x_0q-p(n)y_0)^2 d.
\end{eqnarray*}
Now by the  minimality of $y_0$, we have,
$$
y_0 \leq |x_0q-p(n)y_0|.
$$
\begin{cs}
If $y_0\leq x_0q-p(n)y_0$. Then,
$$
(1+p(n))y_0\leq x_0q.
$$
We use \eqref{e8} and set $p(n)=x_1n+x_2$. This leads to 
\begin{eqnarray*}
-q^2p &\geq& \{(1+x_1n+x_2)^2-q^2d\}y_0^2 \\
& \geq& \{(1+x_1n+x_2)^2-q^2(\alpha n^2 + \beta n +\gamma)\}y_0^2.
\end{eqnarray*}
This implies 
$$
\frac{-q^2p}{y_0^2} \geq (x_1^2-q^2\alpha)n^2+(2x_1+2x_1x_2-q^2\beta)n+(x_2^2+2x_2+1-q^2\gamma),
$$
which contradicts the fact that there are infinitely many $d's$.
\end{cs}
\begin{cs}
If $y_0\leq p(n)y_0-x_0q$.

Then, 
$$
x_0q\leq (p(n)-1)y_0.
$$
We use \eqref{e8} and set $p(n)=x_1n+x_2$. This leads to
\begin{eqnarray*}
-q^2p &\leq& \{(x_1n+x_2-1)^2-q^2d\}y_0^2.\\
-q^2p &\leq& \{(x_1n+x_2-1)^2-q^2(\alpha n^2 + \beta n +\gamma)\}y_0^2.\\
\frac{q^2p}{y_0^2} &\geq& q^2(\alpha n^2 + \beta n +\gamma)-(x_1n+x_2-1)^2.
\end{eqnarray*}
This implies 
$$
(q^2\alpha-x_1^2)n^2+(q^2\beta-2x_1-2x_1x_2)n+(q^2\gamma-x_2^2-2x_2-1)\leq \frac{q^2p}{y_0^2},
$$
which contradicts the fact that there are infinitely many $d's$.
\end{cs}
This completes the proof of Lemma \ref{l1}
\end{proof}
\section{ proof of the Theorem \eqref{t11}}
Let $p\equiv 1\pmod 4$ be a prime. Since $d\equiv 3 \pmod 4$, the Legendre symbol 
$$
\left (\frac{d}{p}\right) = 1. 
$$
This implies, $p$ splits completely in $\mathbb{Q}(\sqrt{d})$, i.e. $p = \mathfrak{P}\mathfrak{P}^{'}$, where $\frak{P}$ and $\frak{P}^{'}$ are prime ideals. We also have  $Nm(\frak{P}) =Nm(\frak{P}^{'}) =p$.
   
Suppose if possible the class number $h$ of $\mathbb{Q}(\sqrt{d})$ is $1$. Then $\frak{P}$ is a principal ideal in $ \mathcal{O}_{\mathbb{Q}(\sqrt{d})}$, say $\frak{P} = \langle a+b\sqrt{d} \rangle $. Thus, we have $$p= Nm(\frak{P}) = Nm_{{\mathbb{Q}(\sqrt{d})}/\mathbb{Q}}(a+b\sqrt{d}) = |a^2-db^2|.$$
This implies
\begin{equation}
a^2-db^2 = \pm p,
\end{equation}
which contradicts the lemma\eqref{l1}.
Hence $h \neq 1$.
Except for finitely many $d's$, we have $\phi(d)>4$, therefore $h|{h_{4d}}^+$ (by Lemma \ref{l2}).
Since $h>1$, we have  ${h_{4d}}^+>1$.

\section{Concluding Remarks}
As mentioned earlier, if we consider the sequence $a_1=1$ for $t=2$, then in this case $d=(z+1)^2-1$. Now for $d$ to be square-free $z$ has to odd. Hence $d$ has to be of the form $(2np)^2-1$, which is precisely the family considered in \cite{Kalyan}. Thus here we get a more generalized and larger family.
\section*{Acknowledgement}
The author would like to thank Dr. Mohit Mishra for going through the article and his continuous support in the completion of this article. The scenic ambience of KSoM and support of the colleagues played an important role in completing the work.


\end{document}